\documentclass[12pt]{amsart}
\usepackage{latexsym,amssymb,amsmath}
 2

\newtheorem{thm}{Theorem}[section]
\newtheorem{lem}[thm]{Lemma}
\newtheorem{prop}[thm]{Proposition}

\newcommand{\thmref}[1]{Theorem~\ref{#1}}
\newcommand{\lemref}[1]{Lemma~\ref{#1}}

\theoremstyle{remark}

\begin{document}

\title[Rankin's method]
{Rankin's method and Jacobi forms of several variables} 
\author{B. Ramakrishnan and Brundaban Sahu}
\address{Harish-Chandra Research Institute, 
        Chhatnag Road, Jhusi,
       Allahabad \linebreak
       211 019,
       India.}
\email[B. Ramakrishnan]{ramki@hri.res.in}
\email[Brundaban Sahu]{sahu@hri.res.in}

\maketitle

\section{Introduction}
There are many interesting connections between differential operators and the 
theory of modular forms and many interesting results have been studied. 
In \cite{rankin}, R. A. Rankin gave a general description of the differential 
operators which send modular forms to modular forms.
In \cite{co}, H. Cohen constructed bilinear operators and obtained elliptic 
modular form with interesting Fourier coefficients. 
In \cite{zag}, D. Zagier studied the algebraic properties of 
these bilinear operators and called them as Rankin-Cohen brackets. In \cite{zag1}, 
following Rankin's method, Zagier computed the $n$-th Rankin-Cohen 
bracket of a modular form $g$ of weight $k_1$ with the Eisenstein series 
of weight $k_2$ and then computed the inner product of this Rankin-Cohen 
bracket with a cusp form $f$ of weight $k = k_1+k_2+2n$ and showed that 
this inner product gives, upto a constant, the special value of the 
Rankin-Selberg convolution of $f$ and $g$. 

\smallskip

Rankin-Cohen brackets for Jacobi forms were studied by Y. Choie 
\cite{{choie1},{choie2}} by using the heat operator. 
Following the work of Zagier mentioned in the above paragraph, Y. Choie and W. Kohnen 
\cite{choie5} generalized the above result of Zagier to Jacobi forms. They 
computed the Petersson scalar product $\langle f,[g,E_{k_2,m_2}]_\nu\rangle$ of a 
Jacobi cusp form $f$ of weight $k$, index $m$ against the 
Rankin-Cohen bracket $[g,E_{k_2,m_2}]_\nu$ of a Jacobi form $g$ of weight $k_1$, 
index $m_1$ and the Jacobi Eisenstein series $E_{k_2,m_2}$ of weight $k_2$, 
index $m_2$, where $k= k_1+k_2+2\nu$ and 
$m= m_1+m_2$.  Although the concept of Rankin-Selberg convolution has not been 
done yet in the case of Jacobi forms, the above mentioned work of Choie and Kohnen 
gives the inner product considered in terms of the special value of a kind of 
Rankin-Selberg type convolution of the Jacobi forms $f$ and $g$. 

\smallskip

In this paper, we generalize the work of Choie and Kohnen to Jacobi forms defined 
over ${\mathcal H} \times {\mathbb C}^{(g, 1)}$. Since the method is similar, we 
shall give only a brief sketch of the proof with the corresponding steps. 

\bigskip

\section{Preliminaries on Jacobi forms over ${\mathcal H} \times 
{\mathbb C}^{(g,1)}$}

\smallskip

Let $g\ge 1$ be a fixed positive integer. The Jacobi group  
$\Gamma^J_{1, g}= \Gamma_1 \ltimes \left({\mathbb Z}^{(g,1)} \times {\mathbb 
Z}^{(g,1)}\right)$ acts on ${\mathcal H} \times {\mathbb C}^{(g,1)}$ in the 
usual way by 
$$ 
\left( \begin{pmatrix} 
               a & b \\
               c & d \\ 
\end{pmatrix}, (\lambda, \mu) \right) \circ(\tau, z)=
\left( \frac {a\tau+b}{c\tau+d}, \frac{z+\lambda \tau+ \mu}{c\tau+d} \right),
$$
where $\Gamma_1 = SL_2({\Bbb Z})$ is the full modular group. 
 
Let $k \in {\mathbb Z}$ and $M$ be a positive definite symmetric half-integral
matrix of size $g \times g$. If $\gamma = \left( 
\begin{pmatrix}
               a & b \\
               c & d \\
               \end{pmatrix}, (\lambda, \mu) \right) \in \Gamma^J_{1,g}$ 
and 
$\phi$ is a complex valued function on ${\mathcal H} \times {\mathbb 
C}^{(g,1)}$,  
then define 
$$ 
\phi|_{k, M} \gamma := (c\tau+d)^{-k} e(-c(c\tau+d)^{-1}M[z+\lambda \tau+ \mu]+
M[\lambda] \tau + 2\lambda^t M z) \phi(\gamma \circ(\tau, z)),
 $$
where $A[B] = B^t A B$ with $A, B$ matrices of appropriate size.

\smallskip

Let $J_{k, M}$ be the space of Jacobi forms of weight $k$ and index
$M$ on  $\Gamma^J_{1, g}$. That is, the space of holomorphic functions 
$\phi: {\mathcal H} \times \mathbb C^{(g,1)} \rightarrow \mathbb C$ 
satisfying
$ \phi|_{k, M} \gamma = \phi$,  $\forall \gamma \in  \Gamma^J_{1, g}$ and 
having a Fourier expansion of the form 
$$ 
\phi(\tau, z)= \sum_{n\in \mathbb Z, r \in {\mathbb Z}^g, 4n \ge M^{-1}[r^t]} 
c(n, r) e(n\tau+rz).
$$   
Further, we say that $F$ is a cusp form if and only if $c(n,r) \not =0$ 
implies 
$4n > M^{-1}[r^t]$. We denote the space of all Jacobi cusp forms by 
$J^{cusp}_{k, M}$.  

\smallskip

For $F, G \in J_{k, M}$ with one of them a Jacobi cusp form, the Petersson 
inner product is defined as 
$$
\langle F, G\rangle = \int_{\Gamma^J_{1, g}\backslash {\mathcal H} \times 
{\mathbb C}^{(g,1)}} F(\tau, z) \overline{G(\tau, z)} v^k e(-4\pi 
M[y] \cdot v^{-1}) ~dV^J_g,
$$
where $\tau= u+iv, z=x+iy$, and  $dV^J_g = v^{-g-2}dudvdxdy$ is the invariant 
measure. The space $(J^{cusp}_{k, M},\langle, \rangle)$ is a finite 
dimensional 
Hilbert space. For more details on Jacobi forms on ${\mathcal H} \times 
{\mathbb C}^{(g,1)}$ we refer to \cite{{boch},{ziegler}}.

\bigskip

\subsection{Poincar\'e series}

Let $n \in \mathbb Z, r \in \mathbb Z^g$, with $4n > M^{-1}[r^t]$. For 
$k>g+2$  let $P_{k, M;(n, r)}$ be the $(n, r)$-th Poincar\'e series in 
$J_{k,M}^{cusp}$ characterized by 

\begin{equation}\label{poincare}
\langle \phi, P_{k, M;(n, r)} \rangle = 
\lambda_{k, M, D} ~c_\phi(n, r) \qquad
{\rm for ~all~} \phi \in J^{cusp}_{k, M}, 
\end{equation}
where $c_\phi(n, r)$ denotes the $(n, r)$-th Fourier coefficient of $\phi$ and 
$$ 
\lambda_{k, M, D}:= 2^{(g-1)(k-g/2-1)-g} \Gamma(k-g/2-1) 
\pi^{-k+g/2+1} (\det M)^{k-(g+3)/2} D^{-k+g/2+1}, 
$$ 
$$
D ~= ~\det(2T), 
~T =\begin{pmatrix}
               n & r/2 \\
               r^t/2 & M \\
               \end{pmatrix}.
$$

\smallskip 

The Poincar\'e series $P_{k, M;(n, r)}$ has the following Fourier
expansion  
$$ 
P_{k, M;(n, r)}(\tau, z)= \!\!\!\!
\sum_{n' \in {\mathbb Z}, r' \in {\mathbb Z}^g, 4n' > M^{-1}[r'^t]} 
\!\!\!(g_{k, M;(n, r)}(n', r')+ (-1)^k g_{k, M;(n, r)}(n', -r'))
~e(n'\tau + r'z),
$$ 
where 
\begin{equation*}
\begin{split}
g_{k, M;(n, r)}(n', r') = \delta_M(n, r, n', r') +  i^{-k} \pi 2^{1-g/2} 
(\det M)^{-1/2} \hskip 5cm& \\
\quad \qquad \qquad \cdot (D'/D)^{k/2-g/4-1/2}\sum_{c \ge 1} H_{M,c}(n, r, n', 
r')
 J_{k-g/2-1} \left(\frac{\pi \sqrt{D'D}}{2^{g-1}\det M \cdot c}\right),& \\
\end{split} 
\end{equation*}
where 
$$
D := \det 2 \begin{pmatrix}
               n' & r'/2 \\
               r'^t/2 & M \\
               \end{pmatrix},
$$ 
\begin{equation*}
\delta_m(n, r, n', r'):= \left\{\begin{array}{ll} 1 & if D=D', r' \equiv r 
\pmod{ {\mathbb Z}^g \cdot 2M}\\
0& otherwise\\
\end{array}\right. 
\end{equation*}
and 
$$ 
H_{M,c}(n, r, n', r') = c^{-g/2-1} \sum_{x({\rm mod~} c),\atop{y({\rm mod~} 
c)^*}} e_c((M[x]+rx+n)y^{-1} + n'y + r'x) e_{2c}(r'M^{-1}r^t)
$$
[In the above, $x$ (resp. $y$) runs over a complete set of representatives 
for 
${\mathbb Z}^{(g,1)}/c {\mathbb Z}^{(g,1)}$ (resp. (${\mathbb Z}/c{\mathbb 
Z})^*$, and $y^{-1}$ denotes an inverse of $y \pmod{c}$, 
$e_c(a):= e^{2 \pi ia/c}$, 
$a \in {\mathbb Z}$, and $J_{k-g/2-1}$ denotes the Bessel function of order 
$k-g/2-1$.]  For details we refer to \cite[Lemma 1]{boch}.

\bigskip

\section{Generalized Heat Operator}
For a positive definite symmetric  half-integral matrix $M = (m_{ij})$ of size 
$g \times g$, we define the heat operator by
\begin{equation}
L_M:= 8 \pi i |M| \frac{\partial}{\partial \tau}- \sum_{1 \leq i, j \le g} 
M_{ij}
 \frac{\partial}{\partial {z_i}}\frac{\partial}{\partial {z_j}},
\end{equation}
where $|M| = \det M$,  $M_{ij}$ is the cofactor of the entry $m_{ij},$ 
$\tau \in \mathcal H, z^t=(z_1, z_2, \cdots, z_g)$ $\in \mathbb C^{g}$.  
Note that when $g=1$ the above heat operator  reduces to the classical heat 
operator, viz., $8 \pi i m \frac{\partial}{\partial \tau} - 
\frac{\partial^2}{\partial {z}^2}$. 

\smallskip

Let $r^t= (r_1, r_2, \cdots, r_g)$. Then using the fact that  
\begin{equation*}
\begin{split}
\frac{\partial}{\partial \tau}\left(e(n\tau + rz) \right) &
= 2\pi i n ~e(n\tau + rz),\\
\frac{\partial}{\partial z_\alpha}\left(e(n\tau + rz)\right) 
&= 2\pi i ~r_\alpha 
~e(n\tau + rz), \quad 1\le \alpha\le g,\\
\frac{\partial}{\partial z_\alpha}\frac{\partial}{\partial z_\beta}
\left(e(n\tau + rz)\right) &= (2\pi i)^2 ~r_\alpha r_\beta
~e(n\tau + rz), \quad 1\le \alpha,\beta\le g,
\end{split}
\end{equation*}

\smallskip

we get 
\begin{equation}
\begin{split} 
L_M(e(n\tau +rz) & =  
~ 8 \pi i |M| \cdot 2\pi in \cdot  
e(n\tau + rz)  -  \sum_{1 \leq \alpha, \beta 
\leq g} M_{\alpha\beta} (2 \pi i)^2 r_\alpha r_\beta~  
e(n\tau + rz)\\
& =~(2 \pi i)^2(4n|M|- \tilde{M}[r^t])~e(n\tau +rz),
\end{split}
\end{equation} 
where $\tilde{A}$ denotes the matrix $(A_{ij})$ with $A_{ij}$ being the 
cofactor of the $ij$-th entry of the symmetric matrix $A$.

\bigskip

We obtain the action of the heat operator on Jacobi forms in the following 
lemma.

\begin{lem} \label{lem:action}
Let $ F \in J_{k, M}.$ Then 
\begin{equation}
 (L_M F)\big\vert_{k+2, M} ~A = L_M(F\big\vert_{k, M}~A) +
(8 \pi i|M|)\left(k-\frac{g}{2}\right)\left(\frac{\gamma}
{\gamma \tau+\delta}\right) (F\big\vert_{k, M} ~A),  
\end{equation}
for all $ A = \begin{pmatrix}
               * & * \\
               \gamma & \delta \\
               \end{pmatrix} \in \Gamma_1 $. 
In general, for any integer $\nu \ge 0$, 
\begin{equation}
\begin{split}\label{matrixaction}
&(L^\nu_M F)|_{k+2\nu, M} A \hskip 12cm \\
&\qquad = \sum^\nu_{l=0} \binom{\nu}{l}(8 \pi i |M|)^{\nu-l}
\frac{(k-g/2+\nu-1)!}{(k-g/2+l-1)!}
 \left(\frac{\gamma}{\gamma\tau+\delta}
 \right) ^{\nu-l} L^l_M(F|_{k, M} A).  
\end{split}
\end{equation}
Moreover, for all $\lambda, \lambda' \in \mathbb Z^g,$
\begin{equation}\label{latticeaction}
 L_M(F|_M[\lambda, \lambda'])=(L_MF)|_M [\lambda, \lambda'].
\end{equation}
\end{lem}

\begin{proof}
Though our $L_M$ operator differs (slightly) from the operator defined in 
\cite{choie4}, the proof goes along the same lines of the proof of 
Lemma 3.3 of \cite{choie4}. 
\end{proof}


We now define the Rankin-Cohen bracket for Jacobi forms on ${\mathcal H}\times 
{\mathbb C}^{(g,1)}$. 

\smallskip

\noindent {\bf Definition:} 
~~Let $\nu \geq 0 $ be an integer and let $F \in J_{k_1, M_1}, G \in 
J_{k_2, M_2}$, where $k_1, k_2$ are positive integers and $M_1$ and $M_2$ are 
positive-definite, symmetric half-integer matrices of size $g\times g$. 
Define the $\nu$-th Rankin-Cohen bracket of $F$ and $G$ by 
\begin{equation}
[F, G]_\nu= \sum^\nu_{l=0}(-1)^l \binom{k_1-\frac{g}{2}+\nu-1}{\nu-l}
\binom{k_2-\frac{g}{2}+\nu-1} {l}|M_1|^{\nu-l}|M_2|^l L^l_{M_1}(F) 
L^{\nu-l}_{M_2}(G)
\end{equation}


Using \lemref{lem:action} we show that the Rankin-Cohen bracket  
$[~,~]_\nu$ gives a bilinear map from $J_{k_1, M_1} \times J_{k_2, M_2}$ 
to $J_{k_1+k_2+2\nu, M_1+M_2}$ (in fact, to $J^{cusp}_{k_1+k_2+2\nu, 
M_1+M_2}$ if $\nu >0$). 


\begin{prop}
Let $\nu \ge 0$ be an integer and let $F \in J_{k_1, M_1}, G \in 
J_{k_2, M_2}$. Then $[F, G]_\nu \in J_{k_1 + k_2 + 2\nu, M_1 + M_2}$. If 
$\nu >0$, then $[F, G]_\nu \in J_{k_1 + k_2 + 2\nu, M_1 + M_2}^{cusp}$. 
\end{prop}

\begin{proof}
By \eqref{latticeaction} we see that the action of the heat operator on Jacobi 
forms is invariant under the lattice action and so the invariance with 
respect to 
the lattice action of the Rankin-Cohen bracket follows from the definition. 
It remains to show that the Rankin-Cohen bracket is invariant, under the 
stroke operation, with respect to the group action. Making use of 
\eqref{matrixaction}, we see that for $A =\begin{pmatrix} * &*\\ c&d\\ 
\end{pmatrix} \in \Gamma_1$

\begin{equation}\label{RC-action}
\begin{split}
&[F,G]_\nu\big\vert_{k_1+k_2+2\nu,M_1+M_2} A \hskip 6cm \\ 
& ~~ = \sum^\nu_{l=0}(-1)^l \binom{k_1'+\nu}{\nu-l}
\binom{k_2'+\nu} {l}|M_1|^{\nu-l}|M_2|^l ~ L^l_{M_1}(F)\big\vert_{k_1+2l, M_1}A 
~ L^{\nu-l}_{M_2}(G)\big\vert_{k_2+2(\nu-l), M_2}A \\
& ~~ = \sum^\nu_{l=0}(-1)^l \binom{k_1'+\nu}{\nu-l} \binom{k_2'+\nu}{l}
\left(\sum_{u=0}^{l}\sum_{v=0}^{\nu-l} \binom{l}{u} 
\binom{\nu-l}{v}(8\pi i)^{\nu-u-v}|M_1|^{\nu-u} |M_2|^{\nu-v} \right.\\
& \hskip 4cm \times ~ \left. \frac{(k_1'+l)!}{(k_1'+u)!} ~
\frac{(k_2'+\nu -l)!}{(k_2'+v)!} \left(\frac{c}{c\tau +d}\right)^{\nu -u-v} 
L_{M_1}^u(F) ~L_{M_2}^{v}(G) \right), \\
\end{split}
\end{equation}
where $k_j' = k_j-\frac{g}{2}-1$, $j=1,2$. When $u+v = \nu$, the right-hand 
side becomes $[F,G]_\nu$, and so it remains to show that the terms 
corresponding 
to $u+v < \nu$ vanish. It is easy to see that the coefficient corresponding to 
$L_{M_1}^u(F) ~L_{M_2}^{v}(G)$, with $u+v\le \nu-1, ~u\le v$ is given by 
$$
\left(\frac{8\pi ic}{c\tau +d}\right)^{\nu-u-v}|M_1|^{\nu-u} |M_2|^{\nu-v} 
~\frac{(k_1'+\nu)!}{u! (k_1'+u)!} \frac{(k_2'+\nu)!}{v! (k_2'+v)!}
\sum_{l =u}^{\nu-v} \frac{(-1)^l}{(l-u)! (\nu-v-l)!} 
$$
and the sum in the last expression is equal to zero. This completes the proof.
\end{proof}

\smallskip

We shall now state the main theorem of this paper.

\begin{thm}\label{main}
Let $ F \in J^{cusp}_{k, M}$ with Fourier coefficients $a(n, r)$ and 
$G \in J_{k_1, M_1}$ with Fourier coefficients $b(n, r)$. Let $E_{k_2, M_2}$ 
be the Jacobi Eisenstein series in $J_{k_2,M_2}$ such  that 
$ k=k_1+k_2+2\nu$ with   $ \nu \geq 0, M=M_1+M_2$ and $k_1>g+2, k_2> k_1+g+2.$
Then 
\begin{equation}\label{main-equ}
\langle F, [G, E_{k_2, M_2}]_\nu\rangle = 
c_{k, k_2, M, M_2, g; \nu} \sum_{n \in {\mathbb Z}, r \in {\mathbb Z}^g,
\atop{ 4n \ge {M}_1^{-1}[r^t]}}
~\frac{(4n|M_1|- \tilde{M}_1[r^t])^\nu ~a(n, r)
\overline{b(n,r)}}{(4n|M|- \tilde{M}[r^t])^{k-g/2-1}},
\end{equation}
where 
\begin{equation}\label{constant}
c_{k, k_2,M, M_2, g; \nu} = 
2^{k'(g-1)-g-2\nu} \pi^{-k'-2\nu} |M|^{k'-1/2} |M_2|^{-\nu} 
\Gamma(k') \frac{\nu!~k_2'!}{(k_2'+\nu)!},
\end{equation}
with $k' = k - g/2 -1$, $k_2' = k_2 - g/2 -1$.
\end{thm}

\smallskip

The rest of this section is devoted to a proof of \thmref{main}.

\bigskip

\subsection{Action of heat operator on Eisenstein series} ~~
Let $M$ be as before and let $E_{k, M}$ be the Jacobi Eisenstein series of
 weight $k$ and index $M$ defined 
by 
\begin{equation}
E_{k, M} = \sum_{\gamma \in \Gamma^J_{1, g, \infty}\backslash \Gamma^J_{1, g}} 
1\big\vert_{k, M} \gamma, 
\end{equation}
where $\Gamma^J_{1, g, \infty} =\left\{ \left( \left( \begin{array}{cc}
               1 & a \\
               0 & 1
               \end{array}\right),(0, \mu)\right)\big\vert a \in {\mathbb Z}, 
               \mu \in \mathbb Z^g \right\}.$

\bigskip

\begin{lem} \label{actionheat}
For a positive integer $\nu$, we have  
\begin{equation}\label{heataction}
\begin{split}
(L^\nu_M E_{k, M})(\tau, z) ~=~ 
(-4)^\nu \frac{\Gamma(k-g/2+\nu)}{\Gamma(k-g/2)}|M|^\nu \hskip 6cm & \\
\quad 
\sum_{ \left(\begin{array}{cc}
               a & b \\
               c & d
               \end{array}\right) \in \Gamma_{1, g,\infty}^J/\Gamma_{1,g}^J 
\lambda \in \mathbb Z } \!\!\!\!\!\!\!
(2\pi ic)^\nu
(c\tau+d)^{-k-\nu} e\left(M[\lambda]\frac{a \tau+b}{c\tau+d} 
+ \frac{2 \lambda^t Mz}{c\tau+d} -\frac{cM[z]}{c\tau+d}\right)
&\\
\end{split}
\end{equation}
\end{lem} 

\begin{proof}
Using the definition of the Eisenstein series, we have 
$$
L^\nu_M E_{k, M}= 
\sum_{\gamma \in \Gamma^J_{1, g, \infty}/ \Gamma^J_{1, g}} 
L^\nu_M\left(1\big\vert_{k, M} \gamma\right). 
$$
By taking $ \left( \left(\begin{array}{cc}
               a & b \\
               c & d
               \end{array}\right), (a \lambda, b \lambda) \right)$ as a set 
of coset representatives in the above sum, where $\left(\begin{array}{cc}
               a & b \\
               c & d
               \end{array}\right) \in \Gamma_1, \lambda \in
{\mathbb Z}^g$,   we have 
$$
 L^\nu_M E_{k, M}= \sum_{ \left(\begin{array}{cc}
               a & b \\
               c & d
               \end{array}\right) \in SL_2(\mathbb Z), \lambda \in
 \mathbb Z^g.} L^\nu_M\left(1\big\vert_{k, M}  \left(\begin{array}{cc}
               a & b \\
               c & d
               \end{array}\right)\right)\big\vert_M(a\lambda, b\lambda). 
$$
It is easy to see that 
\begin{equation*}
\begin{split}
L_M\left(1\big\vert_{k, M}  \left(\begin{array}{cc}
               a & b \\
               c & d
               \end{array}\right)\right)~=~ 
L_M\left((c\tau+d)^{-k} e\left( \frac{-c~M[z]}{c\tau+d}\right)\right) \hskip 
7cm &\\
  =-8 \pi i c|M|(k-g/2) 
(c\tau+d)^{-k-1}e\left( \frac{-c ~M[z]}{c\tau+d}\right), \hskip 4.5cm &\\
\end{split}
\end{equation*}
where we have used the fact that 
\begin{equation*}
\begin{split}
\sum_{1 \le \alpha, \beta \le g} 
M_{\alpha\beta}  \left(\sum_{1 \leq i \leq g} m_{i\beta}z_i \right) 
\left(\sum_{1 \leq i \leq g}m_{i\alpha} z_i \right)  &= 
 ~|M| \sum_{1 \leq \alpha, \beta \le g} m_{\alpha\beta}z_\alpha
z_\beta, \\
\sum_{1 \le \alpha, \beta \le g} M_{\alpha\beta}  m_{\alpha\beta}& = g|M|\\
\end{split}
\end{equation*}
Therefore, 
\begin{equation}
\begin{split}
L^\nu_M\left((c\tau+d)^{-k} e\left( \frac{-c ~M[z]}{c\tau+d}\right)\right) 
\hskip 9cm & \\
~~~= (-4)^\nu \frac{\Gamma(k-g/2+\nu)}{\Gamma(k-g/2)}|M|^\nu(2\pi ic)^\nu 
(c\tau+d)^{-k-\nu} e\left( \frac {-c ~M[z]}{c\tau+d}\right).\hskip 2cm 
\end{split}
\end{equation}
Since,  
\begin{equation}
\begin{split}
(c\tau+d)^{-k-\nu} e\left( \frac
{-c ~M[z]}{c\tau+d}\right)\big\vert_{k, M}
[a \lambda, b \lambda] \hskip 6cm & \\
\quad = (c\tau+d)^{-k-\nu}~
e\left(M[\lambda]\frac{a \tau+b}{c\tau+d}
+ \frac{2 \lambda^t Mz}{c\tau+d} -\frac{cM[z]}{c\tau+d}\right), &
\end{split}
\end{equation}
the required result 
follows.
\end{proof}

\smallskip

\subsection{Representation of $[G, E]_{\nu}$ in terms of the Poincar\'e 
series}~~
We first obtain a growth estimate for the Fourier coefficients of a 
Jacobi form.

\smallskip

\begin{lem}\label{eisen-expan}
Let $k > g+2$ and   $F \in J_{k, M}$ with Fourier coefficients $c(n, r)$. 
Put $D_1 =\sum_{i, j}M_{ij}r_ir_j - 4n |M|$.
Then
\begin{equation}\label{eisen-esti}
c(n, r) \ll |D_1|^{k-g/2-1}, \qquad {\rm ~if~} D_1 <0.
\end{equation}
Moreover, if $F$ is a cusp form, then
\begin{equation}
c(n, r) \ll |D_1|^{k/2-g/2}.  
\end{equation}
In the above the $\ll$ constants depend only on $k, g$ and $|M|.$
\end{lem}

\begin{proof} If $F$ is a cusp form, then the required estimate was 
proved by B\"{o}cherer and Kohnen \cite{boch}. If $F$ is not a cusp form,
then it can be  written as a linear combination of the Eisenstein series $E_{k, 
M}$ and a cusp form. We now show that $e_{k, M}(n, r)$, 
the $(n, r)$-th Fourier coefficient of $E_{k, M}$, satisfies the estimate 
\eqref{eisen-esti}, from which the lemma follows.
Taking the same set of coset representatives as in the proof of the above 
lemma, we get 
$$ 
E_{k, M}= \frac{1}{2} \sum_{c, d \in \mathbb Z, (c, d)=1} \sum_{\lambda
 \in \mathbb Z^g} 1|_{k, M}\left[ \left( \begin{array}{cc}
               a & b \\
               c & d
               \end{array}\right),(a\lambda, b\lambda) \right]
$$
$$
=\frac{1}{2} \sum_{c, d \in \mathbb Z, (c, d)=1} \sum_{\lambda
 \in \mathbb Z^g}(c\tau+d)^{-k} e\left(\frac{-c}{c\tau+d} 
M[z+a\lambda \tau+b \lambda]+M[a \lambda]\tau+2 a \lambda^tMz \right)
$$

$$
=\frac{1}{2} \sum_{c, d \in \mathbb Z, (c, d)=1} \sum_{\lambda
 \in \mathbb Z^g}(c\tau+d)^{-k} e\left(M[\lambda]\frac{a\tau+b}{c\tau+d} 
+ \frac{2 \lambda^tMz}{c\tau+d}-c \frac{M[z]}{c\tau+d} \right)
$$

Proceeding in the usual way by splitting the sum  into two parts as $c=0$ 
and $c \neq 0,$ we see that
$$
E_{k, M}(\tau, z)=\sum_{\lambda \in \mathbb Z^g}e(M[\lambda]\tau+2\lambda^tMz)
+ \sum_{c=1}^{\infty}c^{-k}\sum_{d(c), (d, c)=1}\sum_{\lambda(c)}
e(\frac{a}{c}M[\lambda]) F_{k, M}(\tau+d/c, z-\lambda/c),
$$ 
where $ F_{k, M}(\tau, z)= \sum_{p \in \mathbb Z, q \in \mathbb Z^g}
(\tau+p)^{-k} e\left(\frac{-M[z+q]}{\tau+p} \right).$ Using the Poisson 
summation formula the $(n, r)$-th Fourier coefficient of 
$F_{k, M}(\tau, z)$ is given by 
$$
 \gamma(n, r) = \left\{\begin{array}{ll}
0 & {\rm if~} \tilde{M}[r^t] \geq 4n|M|,\\
\alpha_{k, g}|M|^{-1/2} \left(\frac{2 \pi i}{4|M|}
(4n|M| -\tilde{M}[r^t]\right)^{k-g/2-1}  & {\rm if~}  \tilde{M}[r^t] < 4n|M|.\\
\end{array}\right. 
$$ 
where $\alpha_{k, g}=\displaystyle \left(\frac{1}{2 i}\right)^{g/2} 
\frac{\pi \csc(\pi(k-g/2))}{\Gamma(k-g/2)}. $
Plugging in this Fourier coefficient and estimating the Gauss sum we get
$$
e_{k, M}(n, r) \ll D_1^{k-g/2-1},
$$
where the $\ll$ constant depends only on $k, g$ and $|M|.$
\end{proof}

\smallskip

We need the following lemma which gives the absolute convergence of a series 
which is required to get an expression of the Rankin-Cohen bracket of $F$ 
with the Eisenstein series in terms of the Poincar\'e series. 

\begin{lem} \label{absolute}
The series
$$
v^ke^{-2\pi M[y]/v} \sum_{n\in \mathbb Z, r \in \mathbb Z^g, 
4n \ge M_1^{-1}[r^t],\atop{\gamma \in \Gamma^J_{1, \infty}/\Gamma^J_1}} 
(4n|M_1|-\tilde{M_1}[r^t])^\nu~ b(n,r)~ e(n\tau+ rz)\big\vert_{k, M} \gamma$$
$(\tau=u+iv, z_j=x_j+iy_j, y=(y_1, y_2, \cdots, y_g)^t)$ is absolutely 
uniformly convergent on the subsets $V_{\epsilon, C}=\{ (\tau, z) \in 
\mathcal H  \times \mathbb C^g| v \geq \epsilon, |y_jv^{-1}| \leq C, 
|x_j| \leq 1/\epsilon,
u \leq 1/\epsilon, \forall j=1, 2, \cdots, g\}$ for given $\epsilon >0, C>0$ 
\end{lem}

\begin{proof}
Using \lemref{eisen-expan}, it is sufficient to prove the uniform convergence 
of the series
$$
v^ke^{-2\pi M[y]/v} \sum_{n\in \mathbb Z, r \in \mathbb
 Z^g, 4n \ge M_1^{-1}[r^t], \atop{\gamma \in \Gamma^J_{1, 
\infty}/\Gamma^J_1}} (4n|M_1|- \tilde{M_1}[r^t])^{\nu+k_1-g/2-1} 
|\left(e(n\tau+ rz)\big\vert_{k, M} \gamma\right)(\tau, z)|
$$
in the given ranges. Let $\tau' \in \mathcal H_g$ such that $Z:=\left( 
\begin{array}{cc}
               \tau & z \\
               z & \tau'
               \end{array}\right) \in \mathcal H_{g+1}$.
Also let  $T = \left( \begin{array}{cc}
               n & r^t/2 \\
               r/2 & M
               \end{array}\right)$. Note that by the assumption that 
$4n\ge M_1^{-1}[r^t]$, we see that  $T$ is positive semi-definite. 
Now we embed $\Gamma_{1,g}^J= \Gamma_1 \ltimes (\mathbb Z^g 
\times  \mathbb Z^g) $ into $\Gamma_{g+1}$ (denoted by $\gamma \mapsto 
\gamma^*$) defined by combining the following two embeddings: 
$$
\left(\left(\begin{array}{cc}
               a & b \\
               c & d
               \end{array}\right), (\lambda, \mu)\right) \mapsto 
\left( \left(\begin{array}{cccc}
               a & 0 & b & 0 \\
               0 & I_{g-1} & 0 & 0_{g-1}\\ 
                c & 0 & d & 0 \\
                0 & 0_{g-1} & 0 & I_{g-1}
              \end{array}\right), (\lambda, \mu)\right) 
$$
and 
$$
\left( \left(\begin{array}{cc}
               A & B \\
               C & D
               \end{array}\right), (\lambda, \mu), \right) \mapsto
\left(\begin{array}{cccc}
               A & 0 & B & \mu' \\
               \lambda & 1 & \mu & 0\\ 
                C & 0 & D & -\lambda'\\
                0 & 0 & 0 & 1
               \end{array}\right) $$
where $(\lambda'^t, \mu'^t)=(\lambda, \mu) \left(\begin{array}{cc}
               A & B \\
               C & D
              \end{array}\right)^{-1}$. 
Therefore, we have
$$
\left(e(n\tau+rz)\big\vert_{k, M} \gamma\right)(\tau,z) = 
e(m \tau') \left(e(tr(TZ))\big\vert_k \gamma^*\right)(Z),
$$
where $\big\vert_k$ is the usual stroke operation on functions $F$ defined on 
${\mathcal H}_{g+1}$. We can now view the sum of the absolute terms of the 
$(n, r)$-th Poincar\'e series as a sub-series of the sum of the absolute terms of 
the $T$-th Poincar\'e series on $\Gamma_{g+1}$. 
Let $(\tau, z) \in V_{\epsilon, C}$ for some $\epsilon >0$ and $C >0.$ Then 
by taking 
$\tau'=i \displaystyle{\min_{1 \le j \le g} 
\left\{\left(\frac{y_j^2}{v}+\delta\right)\right\}}$ 
with $\delta > 0,$ we see that  
$Z=\left( \begin{array}{cc}
               \tau & z \\
               z & \tau'
               \end{array}\right) \in {\mathcal H}_{g+1}$ 
with $Y= {\rm Im~} Z > \epsilon'/2$ for some $\epsilon'$ depending on $\epsilon, 
C$ 
and $\delta.$
The sum of the absolute terms of the $T$-th Poincar\'e series on
the subsets $Y \ge \epsilon' I_g, tr(X'X) \le \frac{1}{\epsilon'}$ (up to 
some constants) is majorized by that sum evaluated at an arbitrary single point
$Z_0,$ say $Z_0=iI_g,$ (cf. \cite{{kohnen},{maass}}). So, it is sufficient to 
prove the convergence of  the above series at $(\tau, z) = (i, 0, \cdots, 0)$. 
i.e.,  the convergence of the series (using the the coset representation)
\begin{equation*}
\begin{split}
\sum_{n\in \mathbb Z, r \in \mathbb Z^g, 
4n > M_1^{-1}[r^t],\atop{(c, d)=1, \lambda \in \mathbb Z^g}}
(4n|M_1|- \tilde{M_1}[r^t])^{\nu+k_1-g/2-1}|ci+d|^{-k} 
|e^{2 \pi i (\frac{-c}{c i+d}M[a\lambda i+b \lambda]+ M[a\lambda])}| \hskip 
1.5cm &\\ 
\hskip 5cm \times ~|e^{2 \pi i (n \frac{ai+b}{ci+d} + r\frac{a\lambda i+ b 
\lambda}
{ci+d})}|. &\\
\end{split}
\end{equation*}
Now proceeding as in \cite{choie5} we get the required convergence with the
assumption that $k_2>k_1+g+2.$
\end{proof}

\smallskip

\begin{prop} \label{represent} 
Let $k, k_1, k_2, M, M_1, M_2$ be as in Theorem 3.3. 
Let $G \in J_{k_1, M_1}$ with Fourier expansion 
\begin{equation*}
G(\tau, z)=  \sum_{n\in \mathbb Z, r \in \mathbb Z^g, 
\atop{4n \ge M_1^{-1}[r^t]}} b(n, r) e(n\tau+rz).
\end{equation*} 
Then 
\begin{equation}\label{prop}
[G, E_{k_2, M_2}]_\nu = c_{k_1, k_2,M_1, M_2, g; \nu}
{\displaystyle\sum_{n \in \mathbb Z, r \in 
{\mathbb Z}^g, 4n \ge M_1^{-1}[r^t]}} (4n|M_1|-\tilde{M_1}[r^t])^\nu ~b(n, r)~
P_{k, M; (n, r)},
\end{equation}  
where 
$\displaystyle c_{k_1, k_2,M_1, M_2, g; \nu}=(2\pi)^{-2\nu}|M_2|^{-\nu} 
\frac{v!(k_2-g/2-1)!}{(k_2-g/2+\nu-1)!}$
\end{prop}

\begin{proof} 
Using the definition of the Poincar\'e series, the action of the 
heat operator on Fourier coefficients, and  by the absolute convergence 
(obtained in \lemref{absolute}) we see that the series on the right-hand 
side of 
\eqref{prop} can be written as 
\begin{equation*}
\begin{split}
\sum_{\gamma \in \Gamma^J_{1, g, \infty}\backslash \Gamma^J_{1, g}} 
\left(1\big\vert_{k, M} \gamma\right)(\tau, z) \cdot (2\pi i)^{-2\nu}
(L^\nu_{M_1} G)(\gamma \circ(\tau, z)) \hskip 5cm &\\
\hskip 3.5cm =(2\pi i)^{-2\nu} \sum_{\gamma \in \Gamma^J_{1, g, \infty}\backslash 
\Gamma^J_{1, g}} \left(1\big\vert_{k_2, M_2} \gamma\right)(\tau, z) 
(L^\nu_{M_1} G)\big\vert_{k_1+2\nu, M_1} \gamma(\tau, z). &\\
\end{split}
\end{equation*}
Taking the same set of representatives as in the proof of 
\lemref{actionheat} 
for the sum over $\gamma$ and using \eqref{matrixaction} and the fact that 
$G\in J_{k_1, M_1}$, we get 
\begin{equation*}
\begin{split}
(2\pi i)^{-2\nu} \sum_{\gamma \in \Gamma^J_{1, g, \infty}\backslash
\Gamma^J_{1, g}} \left(1\big\vert_{k_2, M_2} \gamma\right)(\tau, z)
(L^\nu_{M_1} G)\big\vert_{k_1+2\nu, M_1} \gamma(\tau, z) \hskip 6cm &\\
\quad = ~(2\pi i)^{-2\nu} \sum^\nu_{l=0} 4^{\nu-l} 
\frac{(k_1-g/2+\nu-1)!}
{(k_1-g/2+l-1)!} \binom{\nu}{l}|M_1|^{\nu-l}  L^l_{M_1}(G)(\tau, z) \hskip 5cm 
&\\
\!\!\!\!\!\!\!\!\times \!\!\!\!\!\!\!\!\sum_{ 
{\tiny \begin{pmatrix} a&b \\ c& d\\ \end{pmatrix}}, 
\lambda \in \mathbb Z^g } \!\!\!\!\!\!\!\!
\left(\frac{\pi i c}{c\tau+d}\right)^{\nu-l}
\!\!\!\!\!
(c\tau+d)^{-k_2} e\left(
 \frac{-c~ M_2[z+a \lambda  \tau+ b \lambda]}{c\tau+d}
+M_2[a \lambda]+2(a \lambda)^tM_2z \right)\hskip 0.75cm  &\\
\end{split}
\end{equation*}
Using \lemref{actionheat}, the inner sum in the above expression is equal to 
$$
(-4)^{-(\nu-l)} \frac{(k_2-g/2-1)!}{(k_2-g/2+\nu-l-1)!} |M_2|^l L^{\nu-l}_
{M_2} E_{k_2, M_2}(\tau, z).
$$
Therefore, we finally find that the sum on the right-hand side of 
\eqref{prop} equals 
\begin{equation*}
\begin{split}
(2\pi)^{-2\nu}|M_2|^{-\nu} \sum^\nu_{l=0} (-1)^l 
\frac{(k_1-g/2+\nu-1)! (k_2-g/2-1)! } {(k_1-g/2+l-1)! (k_2-g/2+\nu-l-1)!} 
\binom{\nu}{l}|M_1|^{\nu-l} |M_2|^l &\\
\hskip 5cm \times ~L^l_{M_1}(G)(\tau, z) L^{\nu-l}_{M_2} E_{k_2, M_2}(\tau, z).
&\\
\end{split}
\end{equation*}
The proof is now complete.
\end{proof}

\bigskip

\subsection{Proof of \thmref{main}}~~ 
We first observe that by Lemma 3.5 the series on the right hand side of
 \eqref{main-equ} is absolutely convergent and hence is majorized by
$$
\sum_{n \ge 1, r \in \mathbb Z^g \atop{4n \ge M_1^{-1}[r^t]}}
\frac{\left(4n|M_1|-\tilde{M_1}[r^t]\right)^{k_1-g/2-1+\nu}}
{\left(4n|M|-\tilde{M}[r^t]\right)^{k/2-1}} \ll \sum_{n \ge 1}
\frac{n^{g/2}. n^{k_1-g/2-1+\nu}} {n^{k/2-1}}= \sum_{n \ge 1} \frac{1}
{n^\frac{k_2-k_1}{2}}< \infty
$$
after putting $k=k_1+k_2+2\nu$ and by our assumption that $k_2>k_1+g+2$.

The standard fundamental domain for the 
action of $\Gamma_{1, g}^J$ on ${\mathcal H}\times {\mathbb C}^{(g,1)}$ is 
contained in one of the sets $V_{\epsilon, C}$ occurring in the statement of 
\lemref{absolute}, for appropriate $\epsilon$ and $C$. Therefore,  using 
\lemref{absolute} we deduce from \lemref{represent} that 
\begin{equation*}
\langle F, [G, E_{k_2, M_2}]_\nu \rangle =  
c_{k_1, k_2,M_1, M_2, g; \nu} {\displaystyle\!\!\!\!\!\!\sum_{n \in 
{\mathbb Z}, r \in {\mathbb Z}^g, 4n \ge M_1^{-1}[r^t]}} \!\!\!
(4n|M_1|- \tilde{M_1}[r^t])^\nu \overline{b(n, r)} \langle F, P_{k, M; (n, r)}\rangle. 
\end{equation*}
where $c_{k_1, k_2,M_1, M_2, g; \nu}$ is defined as in \eqref{prop}.

Note that $4n > M_1^{-1}[r^t]$ implies $4n > M^{-1}[r^t]$ and hence the
Poincar\'e series $P_{k, M; (n, r)}$ are all cusp forms. On the other hand,
if $4n=M_1^{-1}[r^t],$ $4n \ge M^{-1}[r^t]$ implies $r=0$ and $n=0,$ in
which case one has the Eisenstein series $E_{k, M}.$ Since
$F$ is a cusp form, $ \langle F,E_{k, M} \rangle$ is zero. 
Thus using  \eqref{poincare}, we obtain 

\begin{equation*}
\langle F, [G, E_{k_2, M_2}]_\nu\rangle = 
c_{k, k_2, M, M_2, g; \nu} \sum_{n \in {\mathbb Z}, r \in {\mathbb Z}^g,
\atop{ 4n \ge {M}_1^{-1}[r^t]}}
~\frac{(4n|M_1|- \tilde{M}_1[r^t])^\nu ~a(n, r)
\overline{b(n,r)}}{(4n|M|- \tilde{M}[r^t])^{k-g/2-1}},
\end{equation*}
where $c_{k, k_2,M, M_2, g; \nu}$ is defined as in \eqref{constant}. 
\hfill \qed


\bigskip

\begin{thebibliography}{10}
\bibitem{boch} 
{S. B\"{o}cherer  and  W. Kohnen}, {Estimates for Fourier coefficients of 
Siegel cusp forms}, {\em Math. Ann.} {\bf 297} (1993), 499-517.
\bibitem{choie1}
{Y. Choie}, {Jacobi Forms and the Heat Operator}, {\em Math. Z.} {\bf 1} 
(1997), 
95-101.
\bibitem{choie2}
{Y. Choie}, {Jacobi Forms and the Heat Operator II}, {\em Illinois J. Math.} 
{\bf 42} (1998), 179-186.
\bibitem{choie4}
{Y. Choie and  H. Kim}, {Differential Operators on Jacobi Forms of Several
 Variables}, {\em J. Number Theory.} (2000), 82, 140-163.
\bibitem{choie5}
{Y. Choie and  W. Kohnen}, {Rankin's method and Jacobi forms}, {\em Abh. 
Math. 
Sem. Uni, Hamburg} {\bf 67} (1997), 307-314.
\bibitem{co}
{H. Cohen}, {Sums involving the values at negative integers of 
L-functions of quadratic characters}, {\em Math. Ann.} {\bf 217} (1977), 81-94. 
\bibitem{e-z}
{M. Eichler and D. Zagier}, {The Theory of Jacobi forms}, Prog. in 
Mathematics {\bf 55}, {\em Birkh\"auser} 1985.
\bibitem{kohnen}
{W. Kohnen}, {On the Uniform convergence of Poincar\'e series of exponential
type on Jacobi groups} {\em Abh. Math. Sem. Uni, Hamburg} {\bf 66} (1996), 131-
134.
\bibitem{maass}
{H. Maass}, {\"{U}ber die gleichm\"{a}{\ss}ige Konvergenz der Poincar\'echen 
Reihen}, {\em $n$-ten Grades. Nachr, Akad. Wiss. G\"{o}ttingen} (1964), 
137-144.
\bibitem{rankin}
{R. A. Rankin}, {The Construction of automorphic forms from the derivatives 
of a 
given form}, {\em J. Indian Math. Soc.} {\bf 20} (1956), 103-116. 
\bibitem{rankin1}
{R. A. Rankin}, {The construction of automorphic forms from the derivatives of 
given forms}, {\em Michigan Math. J.} {\bf 4} (1957), 181--186.  
\bibitem{zag}
{D. Zagier}, {Modular forms and differential operators}, {\em Proc. Indian 
Acad. Soc.} {\bf 104} (1994), 57-75.
\bibitem{zag1}
{D. Zagier}, {Modular forms whose Fourier coefficients involve zeta functions
 of 
quadratic fields: Modular Forms of one variable VI}, 
{\em Lecture Notes in Mathematics} {\bf 627} (1977), 106-169. 
\bibitem{ziegler}
{C. Ziegler}, {Jacobi forms of higher degree}, {\em Abh. Math. Sem. Univ. 
Hamburg} 
{\bf 59} (1989), 191-224.
\end{thebibliography}
\end{document}